\documentclass{article}

\usepackage[latin1]{inputenc}
\usepackage{t1enc}
\usepackage{amsmath, amssymb, amsfonts, amsthm, amsopn,epsfig}
\usepackage[none]{hyphenat}
\usepackage{tikz}
\usepackage{float}
\usepackage{graphicx}
\usepackage{caption}
\usepackage[margin=1in]{geometry}
\usepackage[toc,page]{appendix}
\usepackage{epstopdf}
\usepackage{etoolbox}
\usepackage[colorlinks=true]{hyperref}

\theoremstyle{plain}
\newtheorem{theorem}{Theorem}

\newtheorem{corollary}[theorem]{Corollary}
\newtheorem{claim}[theorem]{Claim}
\newtheorem{lemma}[theorem]{Lemma}
\newtheorem{conjecture}[theorem]{Conjecture}

\newtheorem{innercustomthm}{Theorem}

\newenvironment{customthm}[1]
{\renewcommand\theinnercustomthm{#1}\innercustomthm}
{\endinnercustomthm}

\theoremstyle{definition}

\author{Istv\'{a}n Tomon\thanks{\'{E}cole Polytechnique F\'{e}d\'{e}rale de Lausanne, Research partially supported by Swiss National Science Foundation grants no. 200020-162884 and 200021-175977.			
		\emph{e-mail}: \textbf{istvan.tomon@epfl.ch}}
}
\title{Almost tiling of the Boolean lattice with copies of a poset}

\begin{document}
\sloppy
\maketitle

\begin{abstract}
Let $P$ be a partially ordered set. If the Boolean lattice $(2^{[n]},\subset)$ can be partitioned into copies of $P$ for some positive integer $n$, then $P$ must satisfy the following two trivial conditions:
\begin{description}
	\item[(1)] the size of $P$ is a power of $2$,
	\item[(2)] $P$ has a unique maximal and minimal element.
\end{description} Resolving a conjecture of Lonc, it was shown by Gruslys, Leader and Tomon that these conditions are sufficient as well.

 In this paper, we show that if $P$ only satisfies condition (2), we can still almost partition $2^{[n]}$ into copies of $P$. We prove that if $P$ has a unique maximal and minimal element, then there exists a constant $c=c(P)$ such that all but at most $c$ elements of $2^{[n]}$ can be covered by disjoint copies of $P$.
\end{abstract}

\section{Introduction}

The \emph{Boolean lattice} $(2^{[n]},\subset)$ is the power set of $[n]=\{1,...,n\}$ ordered by inclusion. If $P$ and $Q$ are partially ordered sets (posets), a subset $P'\subset Q$ is a \emph{copy} of $P$ if the subposet of $Q$ induced on $P'$ is isomorphic to $P$. A \emph{chain} is a copy of a totally ordered set, and an \emph{antichain} is a copy of a totally unordered set. 

It is an easy exercise to show that every poset $P$ has a copy in $2^{[n]}$ for $n$ sufficiently large. Moreover, if $P$ has a unique maximal and minimal element, then every element of $2^{[n]}$ is contained in some copy of $P$. Therefore, it is natural to ask whether it is possible to partition $2^{[n]}$   into copies of $P$.

The case when $P$ is a chain of size $2^{k}$ was conjectured by Sands \cite{sands}. A slightly more general conjecture was proposed by Griggs \cite{griggs}: if $h$ is a positive integer and $n$ is sufficiently large, then  $2^{[n]}$ can be partitioned into chains such that at most one of these chains have size different from $h$. This conjecture was confirmed by Lonc \cite{lonc}. The author of this paper \cite{me} also established the order of the minimal $n$ for which such a partition exists.

\begin{theorem}\label{chainpartition}
	(\cite{me}) Let $h$ be a positive integer. If $n=\Omega(h^{2})$, the Boolean lattice $2^{[n]}$ can be partitioned into chains $C_{1},...,C_{r}$ such that $2h>|C_{1}|\geq h$ and $|C_{2}|=...=|C_{r}|=h$.
\end{theorem}

We note that Theorem \ref{chainpartition} is stated in \cite{me} as Theorem 5.2 without the additional condition that $|C_{1}|\geq h$. However, the proof of Theorem 5.2 presented in \cite{me} actually gives this insignificantly stronger result, which we shall exploit in this paper.

Lonc \cite{lonc} also proposed two conjectures concerning the cases when $P$ is not necessarily a chain. First, he conjectured that if $P$ has a unique maximal and minimal element and size $2^{k}$, then $2^{[n]}$ can be partitioned into copies of $P$ for $n$ sufficiently large. It is easy to see that these conditions on $P$ are necessary, and it was proved by Gruslys, Leader and the author of this paper \cite{posettiling} that these conditions are sufficient as well.

\begin{theorem}(\cite{posettiling})\label{poset}
	Let $P$ a poset with a unique maximal and minimal element and $|P|=2^{k}$, $k\in \mathbb{N}$. If $n$ is sufficiently large, the Boolean lattice $2^{[n]}$ can be partitioned into copies of $P$.
\end{theorem}

The second conjecture of Lonc targets posets which might not satisfy the conditions stated in Theorem \ref{poset}. In this case, it seems likely that we can still cover almost every element of $2^{[n]}$ with disjoint copies of $P$.

\begin{conjecture}(\cite{lonc})\label{loncconjecture}
	Let $P$ be a poset. If $n$ is sufficiently large and $|P|$ divides $2^{n}-2$, then $2^{[n]}\setminus \{\emptyset,[n]\}$ can be partitioned into copies of $P$.
\end{conjecture}

In \cite{posettiling}, a slightly weaker conjecture is proposed.

\begin{conjecture}(\cite{posettiling})\label{mainconjecture}
	Let $P$ be a poset. There exists a constant $c=c(P)$ such that for every positive integer $n$, there exists $S\subset 2^{[n]}$ such that $2^{[n]}\setminus S$ can be partitioned into copies of $P$ and $|S|\leq c$.
\end{conjecture}

In other words, Conjecture \ref{mainconjecture} asks if we can cover all but a constant number of the elements of $2^{[n]}$ with disjoint copies of $P$. Note that Conjecture \ref{loncconjecture} implies Conjecture \ref{mainconjecture} in the case $|P|$ is not divisible by $4$.

The main result of this manuscript is a proof of Conjecture \ref{mainconjecture} in the case $P$ has a unique maximal and minimal element. We believe that the method of proof presented here might lead to the proof of this conjecture in its full generality. We prove the following theorem.

\begin{customthm}{A}\label{mainthm}
	Let $P$ be a poset with a uniqe maximal and minimal element. There exists a constant $c=c(P)$ such that the following holds. Let $n$ be a positive integer, then there exists $S\subset 2^{[n]}$ such that $2^{[n]}\setminus S$ can be partitioned into copies of $P$, and $|S|\leq c$.
\end{customthm}

This paper is organized as follows. In Section \ref{prelim}, we introduce the main definitions and notation used throughout the paper, and outline the proof of Theorem \ref{mainthm}. In Section \ref{grid}, we prove a modification of Theorem \ref{poset}, where $2^{[n]}$ is replaced with an appropriately sized grid. In Section \ref{grid2}, we show that Conjecture \ref{mainconjecture} holds when $P$ is a grid. In Section \ref{finalproof}, we prove Theorem \ref{mainthm} and conclude with some discussion.

\section{Preliminaries}\label{prelim}

\subsection{Cartesian product of posets}
The proof of Theorem \ref{mainthm} relies heavily on the product structure of $2^{[n]}$. Let us define the \emph{cartesian product} of posets.
If $P_{1},...,P_{k}$ are  posets with partial orderings $\preceq_{1},...,\preceq_{k}$ respectively, the cartesian product $P_{1}\times...\times P_{k}$ is also a poset with partial ordering $\preceq$ such that $(p_{1},...,p_{k})\preceq(p_{1}',...,p_{k}')$ if $p_{i}\preceq_{i} p_{i}'$ for $i=1,...,k$. Also, $P^{k}=P\times...\times P$ is the $k$-th \emph{cartesian power} of $P$, where the cartesian product contains $k$ terms.

For every positive integer $m$, we shall view $[m]$ as a poset with the natural total ordering. For positive integers $a_{1},...,a_{d}$, the cartesian product $[a_{1}]\times...\times [a_{d}]$ is called a \emph{$d$-dimensional grid}; $2$-dimensional grids may be referred to as \emph{rectangles}. That is, if $\preceq$ is the partial ordering on $[a_{1}]\times...\times [a_{d}]$, then $\preceq$ is defined such that if $(x_{1},...,x_{d}),(y_{1},...,y_{d})\in [a_{1}]\times...\times [a_{d}]$, then $(x_{1},...,x_{d})\leq (y_{1},...,y_{d})$ iff $x_{i}\leq y_{i}$ for $i=1,...,d$. Note that the Boolean lattice $2^{[n]}$ is isomorphic to the $n$-dimensional grid $[2]^{n}$ and we might occasionally switch between the two notation without further comments.

\subsection{Outline of the proof}

The framework used in \cite{posettiling} for the proof of Theorem \ref{poset} can be easily modified to show that if $P$ has a unique maximal and minimal element, then Theorem \ref{poset} holds if $2^{[n]}$ is replaced with some appropriately sized grid. We prove the following theorem in Section \ref{grid}.

\begin{theorem}\label{gridpartitionthm}
	Let $P$ be a poset with a unique maximal and minimal element. If $d$ is sufficiently large, then $[2|P|]^{d}$ can be partitioned into copies of $P$.
\end{theorem}

Now our task is reduced to showing that Theorem \ref{mainthm} holds when $P$ is a grid. The new ingredient in our paper is the following result, which shall be proved in Section \ref{grid2}.

\begin{theorem}\label{gridthm}
	Let $P$ be a $d$-dimensional grid. If $n=\Omega(d|P|^{4})$, then there exists $S\subset 2^{[n]}$ such that $2^{[n]}\setminus S$ can be partitioned into copies of $P$ and $|S|\leq (24|P|^{2})^{d}$.
\end{theorem}

These two theorems combined immediately yield our main result, see Section \ref{finalproof}.

\section{Partitioning the grid}\label{grid}

In this section, we prove Theorem \ref{gridpartitionthm}. The proof of this theorem follows the same ideas as the proof of Theorem \ref{poset} in \cite{posettiling}, with slight modifications. We shall reuse some of the results of \cite{posettiling} in order to shorten and simplify this manuscript. 

 \subsection{Weak partitions}
 
 Let $S$ be a finite set and let $\mathcal{F}$ be a family of subsets of $S$. We call a function $w:\mathcal{F}\rightarrow \mathbb{N}$ a \emph{weight function}. If $x\in S$, the \emph{weight of $x$} is the total weight of the sets in $\mathcal{F}$ containing $x$.
 
 Let $t$ be a positive integer. The family $\mathcal{F}$ \emph{contains a $t$-partition}, if there exists a weight function on $\mathcal{F}$ such that the weight of each $x\in S$ is $t$. Similarly, $\mathcal{F}$ \emph{contains a $(1\mod t)$-partition}, if there exists a weight function on $\mathcal{F}$ such that the weight of each $x\in S$ equals to $1\mod t$. 
 
 If $\mathcal{F}$ contains a $1$-partition, then the family of sets having weight $1$ in $\mathcal{F}$ is a partition of $S$ in the usual sense. Also, if $\mathcal{F}$ has a $1$-partition, it trivially has a $t$-partition and a ${(1\mod t)}$-partition as well for every positive integer $t$. 
 
 A remarkable result of Gruslys, Leader and Tan \cite{tiling} is that the existence of a $t$-partition and a ${(1\mod t)}$-partition also implies the existence of a $1$-partition of $S^{n}$, for some sufficiently large $n$, into sets that "look like" elements of $\mathcal{F}$. Let us define this precisely.
 
 For a positive integer $n$, define $\mathcal{F}(n)$ to be the family of all subsets of $S^{n}$ of the form $$\{s_{1}\}\times...\times \{s_{i-1}\}\times F\times \{s_{i+1}\}\times...\times \{s_{n}\},$$
 where $i\in [n]$, $s_{1},...,s_{i-1},s_{i+1},...,s_{n}\in S$ and $F\in \mathcal{F}$. The following result played a key role in both \cite{tiling,posettiling}, and shall play an important role in this paper as well. It was stated explicitly in \cite{posettiling} as Theorem 5.
 
 \begin{theorem}(\cite{posettiling})\label{Spartition}
 	Let $S$ be a finite set, $\mathcal{F}$ be a family of subsets of $S$ and $t\in \mathbb{N}$. Suppose that $\mathcal{F}$ contains a $t$-partition and a $(1\mod t)$-partition. Then for $n$ sufficiently large, $S^{n}$ can be partitioned into elements of $\mathcal{F}(n)$.
 \end{theorem}
 
 Let us say a few remarks about this theorem. Note that  if $S$ is endowed with a partial ordering and $\mathcal{F}$ is the family of all copies of some poset $P$, then $\mathcal{F}(n)$ is also a family of copies of $P$ in $S^{n}$. (However, it is not true that every copy of $P$ in $S^{n}$ is an element of $\mathcal{F}(n)$.) Also, at first glance it might not be clear why is it easier to find a $t$-partition or a $(1\mod t)$-partition than a $1$-partition, but as we shall see in the next section, this problem is substantially simpler.
 
 \subsection{$t$ and $(1\mod t)$-partitions}
 
 In order to prove Theorem \ref{gridpartitionthm}, we only need to show that the family of copies of $P$ in $[2|P|]^{m}$ contains a $t$-partition and a $(1\mod t)$-partition for some $m,t$ positive integers.

 \begin{theorem}\label{tpartition}
 	Let $P$ be a poset with a unique maximal and minimal element. Then there exist positive integers $m$ and $t$ such that the family of copies of $P$ in $[2|P|]^{m}$ contains a $t$-partition.
 \end{theorem}
 
 \begin{proof}
 	In \cite{posettiling}, Lemma 3 states that there exist $m$ and $t$ such that the family of copies of $P$ in $[2]^{m}$ contains a $t$-partition. However, $[2|P|]^{m}$ can be trivially partitioned into copies of $[2]^{m}$: for $(i_{1},...,i_{m})\in [|P|]^{m}$, let 
 	$$B_{i_{1},...,i_{m}}=\{(2i_{1}+\epsilon_{1}-2,...,2i_{m}+\epsilon_{m}-2):(\epsilon_{1},...,\epsilon_{m})\in [2]^{m}\}.$$
 	Then $B_{i_{1},...,i_{m}}$ is isomorphic to $[2]^{m}$ and the family of these sets form a partition of $[2|P|]^{m}$.
 \end{proof}
 
 \begin{theorem}\label{modtpartition}
 	Let $P$ be a poset with a unique maximal and minimal element and let $t$ be a positive integer. Then there exists a positive integer $m$ such that the family of copies of $P$ in $[2|P|]^{m}$ contains a $(1\mod t)$-partition.
 \end{theorem}
 
 \begin{proof}
 	If $|P|=1$, the proof is trivial, so we can suppose that $|P|\geq 2$. 
 	
 	Let $d$ be a positive integer such that $2^{[d]}$ contains a copy of $P$. We show that $m=2d-1$ suffices. For simplicity, write $Q=[2|P|]^{m}$ and let $\preceq$ be the partial ordering on $Q$. 
 	
 	Let $\mathcal{F}$ be the family of copies of $P$ in $Q$. Also, for a subset $A\subset Q$, let $I_{A}:Q\rightarrow \mathbb{N}$ be the indicator function of $A$.
 	
 	 Say that a function  $f:Q\rightarrow \mathbb{Z}$ is \emph{realizable} if there exists a weight function $w:\mathcal{F}\rightarrow \mathbb{N}$ such that $$f(x)\equiv\sum_{\substack{F\in\mathcal{F}\\x\in F}} w(F)\mod t$$
 	for all $x\in Q$.  Note that if $f$ and $g$ are realizable, then both $f+g$ and $f-g$ are realizable, as $f-g\equiv f+(t-1)g\mod t$.
 	
 	Our task is to show that $I_{Q}$ is realizable. 
 	
 	Let $\mathbf{a}$ be the element of $Q$ whose every coordinate is $2$. 
 	\begin{claim}\label{twoelementsclaim}
 	 For every $\mathbf{x}\in Q$, the function $I_{\{\mathbf{x}\}}-I_{\{\mathbf{a}\}}$ is realizable.	
 	\end{claim} 
 
     \begin{proof}
 
 Consider two cases according to the number of $1$ coordinates of $\mathbf{x}$.
 	
 	Case 1: The number of coordinates of $\mathbf{x}$ equal to $1$ is at most $d-1$. Let $J$ be a set of $d$ indices $j\in [m]$ such that $\mathbf{x}(j)\geq 2$. Let 
 	$$R=\{(a_{1},...,a_{m})\in [2]^{m}: a_{j}=1\mbox{, if }j\not\in J\}.$$
 	Then $R$ is isomorphic to $2^{[d]}$, so it contains a copy of $P$. Let such a copy be $A\subset R$. If $\mathbf{b}$ is the maximal element of $A$, then $\mathbf{b}\preceq \mathbf{x}$ and $\mathbf{b}\preceq \mathbf{a}$. Hence, $A_{1}=(A\setminus\{\mathbf{b}\})\cup\{\mathbf{x}\}$ and $A_{2}=(A\setminus\{\mathbf{b}\})\cup\{\mathbf{a}\}$ are both copies of $P$, which implies that the function
 	$$I_{A_{1}}-I_{A_{2}}=I_{\{\mathbf{x}\}}-I_{\{\mathbf{a}\}}$$
 	is realizable.
 	
 	Case 2: The number of coordinates of $\mathbf{x}$ equal to $1$ is at least $d$. In a similar way as in the previous case, we show that $I_{\{\mathbf{x}\}}-I_{\{\mathbf{a}\}}$ is realizable.  Let $J$ be a set of $d$ indices $j\in [m]$ such that $\mathbf{x}(j)= 1$. Let 
 	$$R=\{(a_{1},...,a_{m})\in \{2|P|-1,2|P|\}^{m}: a_{j}=2|P|\mbox{, if }j\not\in J\}.$$
 	Then $R$ is isomorphic to $2^{[d]}$, so it contains a copy of $P$. Let such a copy be $A\subset R$. If $\mathbf{c}$ is the minimal element of $A$, then $\mathbf{x}\preceq \mathbf{c}$ and $\mathbf{a}\preceq \mathbf{c}$. Hence, $A_{1}=(A\setminus\{\mathbf{c}\})\cup\{\mathbf{x}\}$ and $A_{2}=(A\setminus\{\mathbf{c}\})\cup\{\mathbf{a}\}$ are both copies of $P$, which implies that the function
 	$$I_{A_{1}}-I_{A_{2}}=I_{\{\mathbf{x}\}}-I_{\{\mathbf{a}\}}$$
 	is realizable.
 \end{proof}

  The previous claim immediately yields the following claim.
  
  \begin{claim}\label{sum0claim}
  	Let $f:Q\rightarrow \mathbb{Z}$ be a function such that
  	$$\sum_{\mathbf{x}\in Q}f(\mathbf{x})\equiv 0\mod t.$$
  	Then $f$ is realizable.
  \end{claim}

\begin{proof}
    Define the function $g:Q\rightarrow \mathbb{Z}$ such that
	$$g=\sum_{\mathbf{x}\in Q}f(\mathbf{x})(I_{\{\mathbf{x}\}}-I_{\{\mathbf{a}\}}).$$
	Clearly, $g$ is realizable as every term in the sum is realizable by Claim \ref{twoelementsclaim}. Also, $f(\mathbf{x})=g(\mathbf{x})$ for all $\mathbf{x}\neq \mathbf{a}$. But $g$ satisfies the equality $\sum_{\mathbf{x}\in Q}g(\mathbf{x})\equiv 0\mod t$ as well, so we must have $f(\mathbf{a})\equiv g(\mathbf{a})\mod t$. Hence, $f\equiv g\mod t$, which gives that $f$ is realizable.
\end{proof}

Now let $A\subset Q$ be a copy of $P$ and $s=t|Q|/|P|$. Note that $s$ is an integer. The function $sI_{A}$ is realizable and $f=I_{Q}-sI_{A}$ satisfies the equality $\sum_{\mathbf{x}\in Q}f(\mathbf{x})\equiv 0\mod t,$ so $f$ is also realizable by Claim \ref{sum0claim}. But then $I_{Q}=f+sI_{A}$ is realizable as well.
 \end{proof}

  Let us remark that if the family of copies of $P$ in $[2|P|]^{m_{0}}$ contains a $t$-partition or a $(1\mod t)$-partition, then the family of copies of $P$ in $[2|P|]^{m}$ also contains a $t$-partition or $(1\mod t)$-partition, respectively, for $m>m_{0}$, as  $[2|P|]^{m}$ can be trivially partitioned into copies of $[2|P|]^{m_{0}}$.

\subsection{Proof of Theorem \ref{gridpartitionthm}}

Now everything is set to prove that if $P$ has a unique maximal and minimal element, then $[2|P|]^{d}$ can be partitioned into copies of $P$ for $d$ sufficiently large.

\begin{proof}[Proof of Theorem \ref{gridpartitionthm}]
	By Theorem \ref{tpartition}, there exist positive integers $t$ and $m_{1}$ such that the family of copies of $P$ in $[2|P|]^{m_{1}}$ contains a $t$-partition. Also, by Theorem \ref{modtpartition}, there exists a positive integer $m_{2}$ such that the family of copies of $P$ in $[2|P|]^{m_{2}}$ contains a $(1\mod t)$-partition. Let $m=\max\{m_{1},m_{2}\}$, then the family of copies of $P$ in $[2|P|]^{m}$ contains both a $t$ and a $(1\mod t)$-partition. Hence, by Theorem \ref{Spartition}, there exists a positive integer $n$ such that $([2|P|]^{m})^{n}$ can be partitioned into copies of $P$. But then $d\geq mn$ suffices as $([2|P|]^{m})^{n}\cong [2|P|]^{mn}$.
	
\end{proof}

\section{Almost partitioning into grids}\label{grid2}

In this section, we prove Theorem \ref{gridthm}. Let us briefly outline our strategy for the proof of this theorem. Let $P$ be a $d$-dimensional grid. First, we use Theorem \ref{chainpartition} to partition $2^{[m]}$ (for sufficiently large $m$) into chains $C_{1},...,C_{r}$ such that all these chains are large (but still constant sized) and at most one of them, $C_{1}$ might have size not divisible by $|P|$. This chain partition induces a partition of $2^{[dm]}\simeq (2^{[m]})^{d}$ into the family of grids $D_{i_{1},...,i_{d}}\simeq C_{i_{1}}\times...\times C_{i_{d}}$, where $(i_{1},...,i_{d})\in [r]^{d}$. We conclude by showing that unless $i_{1}=...=i_{d}=1$, the grid $D_{i_{1},...,i_{d}}$ can be partitioned into copies of $P$, so the subset of elements which are uncovered by disjoint copies of $P$ is $D_{1,...,1}$, a set of constant size.

Now let us first show that as long as one side of a $d$-dimensional grid $G$ is divisible by $2|P|$, and all the sides of $G$ are large, we can partition $G$ into copies of $P$.

\begin{lemma}\label{gridintogrid}
	Let $P$ be a $d$-dimensional grid and let $c_{1},...,c_{d-1}$ be positive integers satisfying $c_{i}\geq 12|P|^{2}$ for $i\in [d-1]$. Then the $d$-dimensional grid $[2|P|]\times [c_{1}]\times...\times[c_{d-1}]$ can be partitioned into copies of $P$.
\end{lemma}

\begin{proof}
	Let us first prove our theorem for $d=2$, the general statement follows by induction on $d$.
	\begin{claim}\label{rectangle}
		Let $a,b,c$ be positive integers such that $c\geq a^{2}b+2a$ and $a$ is even. Then $[ab]\times [c]$ can be partitioned into copies of $[a]\times [b]$.
	\end{claim}

\begin{proof}
	Write $c=ad+r$, where $d,r$ are positive integers and $0\leq r<a$. If $r=0$, we can trivially partition $[ab]\times [c]$ into copies of $[a]\times [b]$. So let us assume that $r\geq 1$, and let 
	$$\epsilon= \begin{cases} 1 &\mbox{if } r=1 \\
		2 &\mbox{if } r\geq 2. \end{cases}$$
	
	  By the condition $c\geq a^{2}b+2a$, we have $d\geq br+2$. Let $R=[a]\times [b]$ and denote by $\preceq$ the partial ordering on $[ab]\times [c]$.
	
	 Consider the following two collections of copies of $R$ in $[ab]\times [c]$. For $(i,j)\in [a/2]\times [d]$, let 
	 $$A_{i,j}=\{(b(i-1)+x,r+a(j-1)+y:(x,y)\in [b]\times [a]\}$$
	 and
	 $$B_{i,j}=\{(ab/2+b(i-1)+x,a(j-1)+y:(x,y)\in [b]\times [a]\},$$
	 see Figure \ref{image}.
	 Clearly, the sets $A_{i,j}$ and $B_{i,j}$ are copies of $R$ and pairwise disjoint. Also, their union covers all elements of $[ab]\times [c]$, except for the elements of the sets $S=[ab/2]\times [r]$ and $T=[ab/2+1,ab]\times [ad+1,c]$. 
	 
	 Now we shall modify some of the sets $A_{i,j}$ by  switching their maximal elements with an element of $T$, and we shall switch the minimal elements of some of the sets $B_{i,j}$ with an element of $S$. We shall execute these switches in a way that the freed up elements of $[ab]\times [c]$ can be easily partitioned into copies of $R$.
	 
	 For $(i,j)\in [a/2]\times [br]$, let $x_{i,j}=(bi,r+aj)$, which is the maximal element of $A_{i,j}$, and let $y_{i,j}=(ab/2+b(i-1)+1,a(j+\epsilon-1)+1)$, which is the minimal element of $B_{i,j+\epsilon}$. We remark that we used the inequality $j+2\leq br+2\leq d$ to guarantee that $A_{i,j}$ and $B_{i,j+\epsilon}$ exist. Let $\phi$ be any bijection between $[a/2]\times [br]$ and $T$, and let $\varphi$ be a bijection between $[a/2]\times [br]$ and $S$.
	 
	 For $(i,j)\in [a/2]\times [br]$, let 
	 $$A'_{i,j}=(A_{i,j}\setminus \{x_{i,j}\})\cup\{\phi(i,j)\},$$
	 $$B'_{i,j+\epsilon}=(B_{i,j+\epsilon}\setminus \{y_{i,j}\})\cup\{\varphi(i,j)\}.$$
	 Also, for $(i,j)\in [a/2]\times [br+1,d]$, let $A'_{i,j}=A_{i,j}$, and for $(i,j)\in [a/2]\times([\epsilon]\cup [br+\epsilon+1,d])$, let $B'_{i,j}=B_{i,j}$. 
	 For $(i,j)\in [a/2]\times [br]$, the set $A_{i,j}$ is contained in $[ab/2]\times [abr+r]$. As $abr+r<ad+1$, we have that every element of $T$ is $\preceq$-larger than every element of $A_{i,j}$. Hence, $A'_{i,j}$ is also a copy of $R$, remembering that $x_{i,j}$ is the maximal element of $A_{i,j}$. Similarly, $B'_{i,j+\epsilon}$ is also a copy of $R$, as $B'_{i,j+\epsilon}$ is equal to $B_{i,j}$ with its minimal element replaced with a $\preceq$-smaller element.
	 For $(i,j)\in [a/2]\times [d]$, the sets $A'_{i,j}$ and $B'_{i,j}$ are disjoint copies of $R$ and their union covers all elements of $[ab]\times [c]$, except for the set $$X=\{x_{i,j}:(i,j)\in [a/2]\times [br]\}\cup\{y_{i,j}:(i,j)\in [a/2]\times [br]\}.$$

     But $X$ can be easily partitioned into copies of $R$. For $k\in [r]$, let 
	 
	 $$C_{k}=\{x_{i,b(k-1)+j}:(i,j)\in [a/2]\times [b]\}\cup \{y_{i,b(k-1)+j}:(i,j)\in [a/2]\times [b]\}.$$
	 Clearly, the sets $C_{k}$ for $k=1,..,r$ partition $X$, so the only thing left is to show that $C_{k}$ is a copy of $R$. 
	 
	  We show that the function $\pi:C_{k}\rightarrow R$ defined by $\pi(x_{i,b(k-1)+j})=(i,j)$ and $\pi(y_{i,b(k-1)+j})=(i+a/2,b)$ is an isomorphism.
	  As $\pi$ defines an isomorphism between $\{x_{i,b(k-1)+j}:(i,j)\in [a/2]\times [b]\}$ and $[a/2]\times [b]$, and it defines an isomorphism between $\{y_{i,b(k-1)+j}:(i,j)\in [a/2]\times [b]\}$ and $[a/2+1]\times [b]$, it is enough to show that $x_{i,b(k-1)+j}\preceq y_{i',b(k-1)+j'}$ if and only if $j\leq j'$. But we have 
	  $$x_{i,b(k-1)+j}=(bi,r+a(b(k-1)+j))$$
	  and
	  $$y_{i',b(k-1)+j'}=(ab/2+b(i'-1)+1,a(b(k-1)+j'+\epsilon-1)+1),$$
	  so $x_{i,b(k-1)+j}\preceq y_{i',b(k-1)+j'}$ holds if and only if $r+aj\leq a(j'+\epsilon-1)+1$. By the choice of $\epsilon$, this inequality is equivalent to $j\leq j'$, finishing our proof.
	 
	 \begin{figure}
	 	\begin{center}
	 		\includegraphics[scale=1]{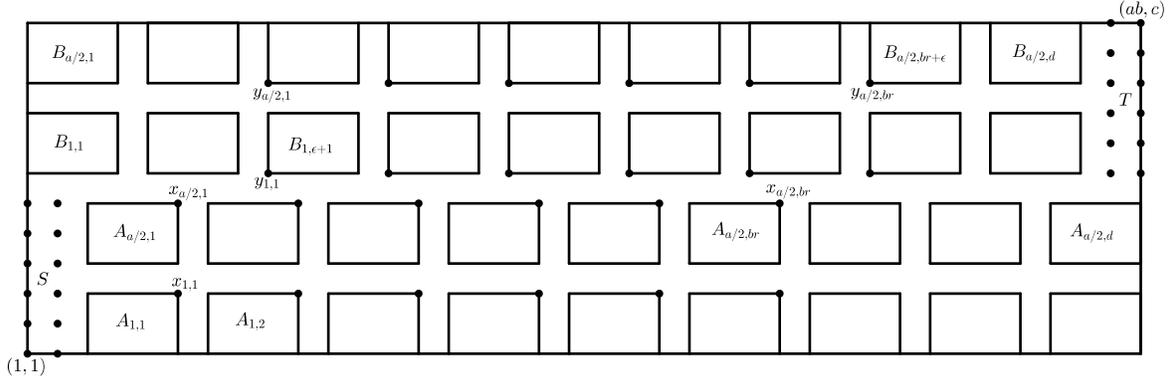}
	 	\end{center}
	 	\caption{Partitioning $[ab]\times [c]$ into copies of $R$}
	 	\label{image}
	 \end{figure}
	  
\end{proof}

 First, we show that if $P=[a_{1}]\times...\times [a_{d}]$, $a_{1}$ is even, $\alpha=|P|=a_{1}...a_{d}$, $c_{1},...,c_{d-1}>3\alpha^{2}$, then the grid $G=[\alpha]\times [c_{1}]\times...\times [c_{d-1}]$ can be partitioned into copies of $P$. 
 
 Let us proceed by induction on $d$. If $d=1$, the statement is trivial. Let us assume that $d\geq 2$ and the statement holds for all grids of dimension $d-1$. Let $a'_{1}=a_{1}a_{2}$ and $a'_{i}=a_{i+1}$ for $i\in [2,d-1]$. Define the $(d-1)$-dimensional grid $P'=[a'_{1}]\times...\times[a'_{d-1}]$. We have that $a'_{1}$ is even, $\alpha=a'_{1}...a'_{d-1}$ and $c_{1},...,c_{d-2}>3\alpha^{2}$. Hence, by our induction hypothesis, the grid $G'=[\alpha]\times [c_{1}]\times...\times [c_{d-2}]$ can be partitioned into copies of $P'$. But this gives a partition of $G=G'\times [c_{d-1}]$ into copies of $P'\times [c_{d-1}]$. Hence, it is enough to show that $P'\times [c_{d-1}]$ can be partitioned into copies of $P$. 
 
 As  $a_{1}$ is even and $c_{d-1}\geq3\alpha^{2}\geq a_{1}^{2}a_{2}+2a_{1}$, an immediate application of Claim \ref{rectangle} shows that the rectangle $R=[a_{1}a_{2}]\times [c_{d-1}]$ can be partitioned into copies of $R'=[a_{1}]\times [a_{2}]$. But then $P'\times [c_{d-1}]\cong R\times [a_{3}]\times...\times [a_{d}]$ can be partitioned into copies of $P=R'\times [a_{3}]\times....\times [a_{d}]$ as well.
 
 To finish our proof, let us assume that $P=[a_{1}]\times...\times [a_{d}]$, where $a_{1}$ is not necessarily even, and $c_{1},...,c_{d-1}\geq 12|P|^{2}$. Let $P'=[2a_{1}]\times...\times [a_{d}]$. We have $|P'|=2|P|$, so by the previous argument, we have that $[|P'|]\times [c_{1}]\times...\times[c_{d-1}]$ can be partitioned into copies of $P'$. But $P'$ can be trivially  further partitioned into two copies of $P$.

\end{proof}

We shall use the following immediate corollary of the previous theorem.

\begin{corollary}\label{gridcor}
	Let $P$ be a $d$-dimensional grid and let $c_{1},...,c_{d}$ be positive integers satisfying $c_{i}\geq 12|P|^{2}$ for $i\in [d]$, and suppose that at least one of $c_{1},...,c_{d}$ is divisible by $2|P|$. Then the $d$-dimensional grid $ [c_{1}]\times...\times[c_{d}]$ can be partitioned into copies of $P$.
\end{corollary}

\begin{proof}
Without loss of generality, suppose that $c_{d}$ is divisible by $2|P|$. Then we can trivially partition $[c_{1}]\times...\times [c_{d}]$ into copies of $[2|P|]\times [c_{1}]\times...\times [c_{d-1}]$. But $[2|P|]\times [c_{1}]\times...\times [c_{d-1}]$ can be partitioned into copies of $P$ by Lemma \ref{gridintogrid}.
\end{proof}

Now we are ready to prove that we can almost partition $2^{[n]}$ into copies of a grid.

\begin{proof}[Proof of Theorem \ref{gridthm}]
Let $h=12|P|^{2}$. By Theorem \ref{chainpartition}, there exists a constant $c$ such that if $m\geq ch^{2}$, then  $2^{[m]}$ can be partitioned into chains $C_{1},...,C_{r}$ such that $h\leq |C_{1}|\leq 2h$ and $|C_{2}|=...=|C_{r}|=h$. Suppose that $n\geq cdh^{2}$, then we can find positive integers $m_{1},...,m_{d}$ such that $n=m_{1}+...+m_{d}$ and $m_{i}\geq ch^{2}$ for $i\in [d]$. Then $2^{[m_{i}]}$ has a partition into  chains $C_{i,1},...,C_{i,r_{i}}$ such that $h\leq |C_{i,1}|\leq 2h$ and $|C_{i,2}|=...=|C_{i,r_{i}}|=h$. 

 The sets $C_{1,j_{1}}\times...\times C_{d,j_{d}}$ for $(j_{1},...,j_{d})\in [r_{1}]\times...\times [r_{d}]$ partition $2^{[m_{1}]}\times...\times 2^{[m_{d}]}$. Hence, as $2^{[n]}\cong 2^{[m_{1}]}\times...\times 2^{[m_{d}]}$, the Boolean lattice $2^{[n]}$ also has a partition into sets $B_{j_{1},...,j_{d}}$ for $(j_{1},...,j_{d})\in [r_{1}]\times...\times [r_{d}]$, where $B_{j_{1},...,j_{d}}$ is isomorphic to $C_{1,j_{1}}\times...\times C_{d,j_{d}}$. But then $B_{j_{1},...,j_{d}}$ is a $d$-dimensional grid isomorphic to $[|C_{1,j_{1}}|]\times...\times[|C_{d,j_{d}}|]$, where $|C_{i,j_{l}}|\geq 12|P|^{2}$ for $(i,l)\in [d]\times [d]$. If $(j_{1},...,j_{d})\neq (1,...,1)$, then we also have that at least one of $|C_{j_{1}}|,...,|C_{j_{d}}|$ is divisible by $2|P|$. Thus, applying Corollary \ref{gridcor}, $B_{j_{1},...,j_{d}}$ can be partitioned into copies of $P$ unless $(j_{1},...,j_{d})\neq (1,...,1)$. Setting $S=B_{1,...,1}$, we get that  $2^{[n]}\setminus S$ can be partitioned into copies of $P$. As 
 $$|S|=|C_{1,1}|...|C_{d,1}|\leq 2^{d}h^{d}=(24|P|^{2})^{d},$$
 our proof is finished.
\end{proof}

\section{Conclusion}\label{finalproof}

In this section, we finish the proof of our main theorem and conclude with some discussion. 

Let us remind the reader of the statement of our main theorem. We shall prove that if the poset $P$ has a unique maximal and minimal element, then all but a constant number of elements of $2^{[n]}$ can be covered by disjoint copies of $P$.

\begin{proof}[Proof of Theorem \ref{mainthm}]
	By Theorem \ref{gridpartitionthm}, there exists  a positive integer $d$ such that $[2|P|]^{d}$ can be partitioned into copies of $P$. Also, by Theorem \ref{gridthm}, there exists $n_{0}=\Omega(d(2|P|)^{4d})$ such that for $n>n_{0}$, all but at most $24^{d}(2|P|)^{2d^{2}}$ elements of $2^{[n]}$ can be covered by disjoint copies of $[2|P|]^{d}$.
	
	Setting $c(P)=\max\{2^{n_{0}},24^{d}(2|P|)^{2d^{2}}\}$, we get that all but at most $c(P)$ elements of $2^{[n]}$ can be covered by disjoint copies of $P$ for all positive integer $n$.
\end{proof}

Say that a poset $Q$ is \emph{connected} if its comparability graph is connected. The following generalization of Theorem \ref{mainthm} can be easily proved following the same line of ideas.  Let us only sketch its proof.

\begin{customthm}{B}
	Let $P,Q$ be posets such that $P$ has a unique maximal and minimal element, and $Q$ is connected. For every positive integer $n$, there exists a constant $c=c(P,Q)$ such that all but at most $c$ elements of $Q^{n}$ can be covered by copies of $P$.
	
	Also, if every prime divisor of $|P|$ also divides $|Q|$, then $Q^{n}$ can be partitioned into copies of $P$ for $n$ sufficiently large.
\end{customthm}

\begin{proof}[Sketch proof.]
	We are done if we prove the following generalization of Theorem \ref{gridthm}.
	
	\begin{customthm}{\protect\NoHyper\ref{gridthm}\protect\endNoHyper$^{+}$}
		Let $P$ be a grid and let $Q$ be a connected poset. Then there exists  $c=c(P,Q)$ such that the following holds. For every positive integer $n$, there exists $S\subset Q^{n}$ such that $Q^{n}\setminus S$ can be partitioned into copies of $P$ and $|S|\leq c$.
		
		Also, if every prime divisor of $|P|$ also divides $|Q|$, then $Q^{n}$ can be partitioned into copies of $P$ for $n$ sufficiently large.
	\end{customthm}
	
	The proof of this theorem easily follows from the combination of Corollary \ref{gridcor} and from the following generalization of Theorem \ref{chainpartition}, which appeared as Theorem 5.1 in \cite{me}.
	
	\begin{customthm}{\protect\NoHyper\ref{chainpartition}\protect\endNoHyper$^{+}$}(\cite{me}) Let $Q$ be a connected poset and let $c$ be a positive integer. If $n$ is sufficiently large, then $Q^{n}$ can be partitioned into chains $C_{1},...,C_{r}$ such that $c\leq |C_{1}|< 2c$ and $|C_{2}|=...=|C_{r}|=c$.
    \end{customthm}
    
    Again, this theorem is stated in \cite{me} without the additional condition that $c\leq |C_{1}|< 2c$, but the proof appearing in \cite{me} yields this slightly stronger result.
\end{proof}

Conjecture \ref{mainconjecture} is still open in the case $P$ does not have a unique maximal or minimal element. However, we believe that a similar approach as the one presented in this paper may lead to its proof in its full generality. Let us propose such an approach.

Let $S_{2k}$ denote the poset on $2k$ elements with partial ordering $\preceq$, whose elements can be partitioned into two $k$ element sets $A_{0},A_{1}$ such that $A_{0}$ and $A_{1}$ are antichains, and every element of $A_{0}$ is $\preceq$-smaller than every element of $A_{1}$.

\begin{enumerate}
	\item Show that for fixed $k$, if $n$ is sufficiently large, then $2^{[n]}$ can be partitioned into copies of $S_{2k}$ and a chain of size at least $k$.
	
	\item\label{step2} Show that if $P$ is a poset, then there exist positive integers $m$ and $t$ such that the family of copies of $P$ in $S_{2|P|}^{m}$ has a $t$-partition and a $(1\mod t)$-partition. This would  imply that $S_{2|P|}^{n}$ can be partitioned into copies of $P$ for $n$ sufficiently large.
	
	\item Show that for positive integers $k$ and $l$, if the positive integers $a$ and $b$ are sufficiently large, then $S_{4kl}\times [a]\times [b]$ can be partitioned into copies of $S_{2k}\times S_{2l}$. This would imply a statement similar to Lemma \ref{gridintogrid}, saying that the cartesian product $S_{2^{d}k_{1}...k_{d}}\times [a_{1}]\times..\times [a_{2d-2}]$ can be partitioned into copies of $S_{2k_{1}}\times...\times S_{2k_{d}}$, if $a_{1},...,a_{2d-2}$ are sufficiently large.
\end{enumerate}

These three steps, if proven, imply Conjecture \ref{mainconjecture} in a similar fashion as the combination of Theorem \ref{chainpartition}, \ref{gridpartitionthm}, \ref{gridthm}  imply Theorem \ref{mainthm}. The advantage of this approach is that after proving step \ref{step2}, which seems to be the most straightforward of the three steps, we can forget about the poset $P$ and we can work with the family of posets $\{S_{2k}\}_{k=1,2,...}$ instead.

\end{document}